\documentclass[11pt]{article}

\usepackage[top=1in, left=1.25in, bottom=1in, right=1.25in]{geometry}

\usepackage{hyperref}  
\usepackage{setspace} 

\usepackage{graphicx}
\usepackage{amsfonts}
\usepackage{amssymb}
\usepackage{amsmath}
\usepackage{amsthm}

\newtheorem{theorem}{Theorem}[section]
\newtheorem{corollary}[theorem]{Corollary}
\newtheorem{lemma}[theorem]{Lemma}

\newcommand{\del}{\backslash}

\title
{Graphical representations of graphic frame matroids}
\author{
   Rong Chen\thanks{Fuzhou University, rongchen@fzu.edu.cn. Supported by CNNSF (No.11201076), SRFDP (No.20113514120010), JA11032 and CSC.}
\and
   Matt DeVos\thanks{Simon Fraser University, mdevos@sfu.ca.
     Supported in part by an NSERC Discovery Grant (Canada)}
\and
   Daryl Funk\thanks{Simon Fraser University, dfunk@sfu.ca. Supported in part by an NSERC Postgraduate Scholarship.}
\and
   Irene Pivotto\thanks{University of Western Australia, irene.pivotto@uwa.edu.au. Supported by an Australian Research Council Discovery Project (project number DP110101596).}
}

\begin{document}

\maketitle

\begin{abstract}
A frame matroid $M$ is \emph{graphic} if there is a graph $G$ with
cycle matroid isomorphic to $M$. In general, if there is one such
graph, there will be many. Zaslavsky has shown that frame matroids are precisely those having a
representation as a biased graph; this class includes graphic
matroids, bicircular matroids, and Dowling geometries. Whitney characterized which graphs have isomorphic cycle
matroids, and Matthews characterized which graphs have isomorphic graphic bicircular matroids. In this paper, we give
a characterization of which biased graphs give rise to isomorphic graphic frame matroids.
\end{abstract}

\section{Introduction}
A \emph{biased graph} $\Omega$ consists of a pair $(G,
\mathcal{B})$, where $G$ is a graph and $\mathcal{B}$ is a
collection of cycles of $G$, called \emph{balanced}, obeying the
\emph{theta property}. A \emph{theta} graph consists of a pair of
distinct vertices and three internally disjoint paths between them;
the \emph{theta property} is the property that no theta subgraph
contains exactly two balanced cycles. Cycles not in $\mathcal{B}$
are called \emph{unbalanced}. We write $\Omega = (G, \mathcal{B})$
and say $G$ is the \emph{underlying graph} of $\Omega$. Throughout
graphs are finite, and may have loops and parallel edges.

Biased graphs were introduced by Zaslavsky in \cite{Zaslavsky-I},
and in \cite{Zaslavsky-II} Zaslavsky defined a natural matroid with
ground set the edges of a biased graph $(G, \mathcal{B})$, which we
may describe in terms of its circuits as follows. A set $C \subseteq
E(G)$ is a circuit in this matroid if in $(G, \mathcal{B})$, $C$
induces one of: a balanced cycle, two edge-disjoint unbalanced
cycles intersecting in only one vertex, two vertex-disjoint
unbalanced cycles along with a path connecting them, or a theta
subgraph with all cycles unbalanced.

A matroid is \emph{frame} if it may be extended such that it
possesses a basis $B_0$ (a frame) such that every element is spanned
by at most two elements of $B_0$. Zaslavsky \cite{Zaslavsky-Frame}
has shown that the class of frame matroids is precisely that of
matroids arising from biased graphs as described above (whence these
have also been called \emph{bias} matroids). Given a biased graph
$\Omega = (G, \mathcal{B})$ we denote by $F(\Omega)$ or $F(G,
\mathcal{B})$ the frame matroid arising from $\Omega$. Observe that
given a graph $G$, if $\mathcal{B}$ contains all cycles in $G$, then
$F(G, \mathcal{B})$ is the cycle matroid $M(G)$ of $G$, and that
$F(G, \emptyset)$ is the bicircular matroid of $G$. Frame matroids
also include Dowling geometries \cite{Dowling} (see also, for
example, \cite{MR2459453} and \cite{Zaslavsky-I}).

Whitney~\cite{Whitney33} characterised when two graphs give rise to the same graphic matroid, and Matthews \cite{Matthews}  characterized which graphs have isomorphic graphic bicircular matroids. To state Whitney's result we first need some definitions.
Given a graph $H$ and a set of edges $Y$, we let $H|Y$ denote the subgraph of $H$ with edge set $Y$ and no isolated
vertices.
Let $H$ be a graph, and let $(X_1, X_2)$ be a partition of $E(H)$
such that $V(H|X_1) \cap V(H|X_2) = \{u_1, u_2\}$. We say that $H'$
is obtained by a \emph{Whitney flip} of $H$ on $\{u_1,u_2\}$ if $H'$
is a graph obtained by identifying vertices $u_1,u_2$ of $H|X_1$
with vertices $u_2,u_1$ of $H|X_2$, respectively.
A graph $H'$ is \emph{2-isomorphic to} $H$ if $H'$ is obtained from
$H$ by a sequence of the operations: Whitney flips, identifying two
vertices from distinct components of a graph, or partitioning a
graph into components each of which is a block of the original
graph.
\begin{theorem}[Whitney's $2$-Isomorphism Theorem, \cite{Whitney33}]
Let $G$ and $H$ be graphs without isolated vertices. Then $M(G)\cong
M(H)$ if and only if $G$ and $H$ are $2$-isomorphic.
\end{theorem}

Six families of biased graphs whose frame matroids are graphic are
defined and exhibited in Section~\ref{sec:SixFamilies}. A biased
graph in any of these families is obtained from a graph $G$ by a
simple operation, and it is easily checked that the frame matroid
arising from the resulting biased graph is isomorphic to the cycle
matroid of $G$. For ease of reference, we name them: (1) balanced,
(2) fat thetas, (3) curlings, (4) pinches, (5) 4-twistings, and (6)
consecutive odd-twistings. We call the corresponding operation in
each case by the same name. Our main result says that every graphic
frame matroid comes from a biased graph in one of these families.

\begin{theorem}\label{2-con}
Let $G$ be a 2-connected graph and $\Omega$ a biased graph with
$F(\Omega) = M(G)$. Then there is a graph $H$ 2-isomorphic to $G$
such that either $\Omega$ is balanced with underlying graph $H$, or
$\Omega$ is obtained from $H$ as a fat theta, a curling, a pinch, a
4-twisting, or a consecutive odd-twisting.
\end{theorem}

Increasing the connectivity of the graph $G$ in Theorem~\ref{2-con}
reduces the possible biased graph representations of $F(\Omega)$.
Asking that $G$ be 3-connected removes one family from the list of
possibilities, and simplifies the curling operation. The following
is an immediate consequence of Theorem~\ref{2-con}.


\begin{corollary}\label{3-con}
Let $G$ be a $3$-connected graph with at least four vertices and $\Omega$ a biased graph with
$F(\Omega)=M(G)$. Then either $\Omega$ is balanced with underlying
graph $G$, or obtained from $G$ as a simple curling, a pinch, a
4-twisting, or a consecutive odd-twisting.
\end{corollary}

If we further demand that $G$ be 4-connected, we are still left with
two possible families of biased graph representations. The following
is an immediate consequence of Corollary~\ref{3-con}.

\begin{corollary}\label{4-con}
Let $G$ be a $4$-connected graph with at least five vertices and
$\Omega$ a biased graph with $F(\Omega)=M(G)$. Then either $\Omega$
is balanced with underlying graph $G$, or obtained from $G$ as a
simple curling or a pinch.
\end{corollary}

Finally, we note that for any $k \geq 4$, if $G$ is a $k$-connected
graph on $n$ vertices, then in general there may be up to ${n
\choose 2} + n$ non-isomorphic biased graphs $\Omega$ with
$F(\Omega)$ isomorphic to $M(G)$ (obtained as pinches and simple
curlings of $G$). Corollary~\ref{4-con} says, however, that these
will be all the biased graph representations of $M(G)$.

The remainder of this paper is organized as follows. First we
exhibit the six families of biased graphs whose frame matroids are
graphic (appearing in Theorem~\ref{2-con}). We then show that the
frame matroids arising from these biased graphs are indeed graphic,
and that every graphic frame matroid arises from a biased graph in
one of these families.

\section{Six families of biased graphs with graphic frame matroids} \label{sec:SixFamilies}

We now describe six families of biased graphs whose frame matroids
are graphic. For any positive integer $n$, set
$[n]=\{1,2,\cdots,n\}$.

\smallskip
\emph{1.\ Balanced biased graphs}. \ Let $\Omega = (H, \mathcal{B})$
be a biased graph. If every cycle of $H$ is in $\mathcal{B}$, then
$\Omega$ is \emph{balanced}. Clearly, $F(\Omega) = M(H)$, so
$F(\Omega)$ is graphic.

\smallskip
\emph{2.\ Fat thetas}. \ To describe our second family, let $H_1,
H_2, H_3$ be non-empty graphs with distinct vertices $x_i, y_i \in
V(H_i)$. Let $H$ be obtained from $H_1, H_2, H_3$ by identifying
$y_i$ and $x_{i+1}$  to a vertex  $w_i$ for every $i\in[3]$ (where
the indices are modulo $3$; see the left of
Figure~\ref{fig:fat_theta}).
\begin{figure}[htbp]
\begin{center}
\includegraphics[page=1,height=4cm]{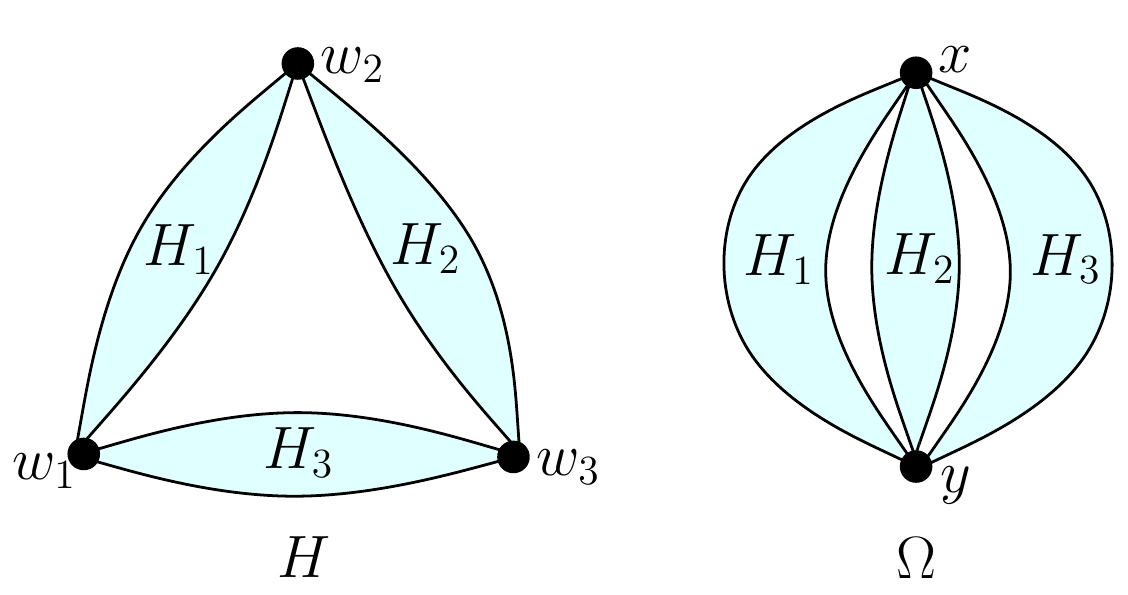}
\end{center}
\caption{A fat theta.} \label{fig:fat_theta}
\end{figure}
Let $\Omega=(\Gamma,\mathcal{B})$ be a biased graph, where $\Gamma$
is obtained from $H_1,H_2,H_3$ by identifying $x_1, x_2, x_3$ to a
vertex $x$ and identifying $y_1, y_2, y_3$ to a vertex $y$. A cycle
of $\Gamma$ is in $\mathcal{B}$ if and only if $E(C)$ is completely
contained in one of $H_1, H_2$ or $H_3$ (see the right of
Figure~\ref{fig:fat_theta}). Then we say that $\Omega$ is a
\emph{fat theta} obtained from $H$.

\smallskip

A biased graph $\Omega$ is a \emph{signed graph} if its edges can be
labelled by $1$ or $-1$ such that a cycle $C$ is balanced in
$\Omega$ if and only if $E(C)$ contains an even number of edges
labelled $-1$. In all figures of signed graphs we adopt the
following convention. A shaded area around a vertex denotes that all
the edges in that area incident with that vertex are labelled with
$-1$. Bold edges are also labelled $-1$. All unmarked edges are
labelled $1$.

\smallskip
\emph{3.\ Curlings}.  \
Let $H$ be a 2-connected graph, $v \in V(H)$, and
suppose that there are distinct vertices $v_1, \ldots, v_k$ and
connected subgraphs $H_1, \ldots, H_k$ of $H$  such that $V(H_i)\cap
V(H | (E(H)-E(H_i))) = \{v,v_i\}$.
Suppose moreover that every edge incident with $v$ is contained in some
$H_i$. Let $\Omega$ be the signed graph obtained from $H$ by first
labeling all edges incident with $v$ by $-1$ and then changing any
such edge $e=uv$ to $uv_i$ when $u\in H_i$ (if $u=v_i$ this produces
a loop at $v_i$), and keeping all other edges not incident with $v$
unchanged and labelled by $1$ (see Figure~\ref{fig:curling}). Then
we say that $\Omega$ is a \emph{curling} of $H$. If, for every $i$, every edge in $H_i$ is between $v$
and $v_i$ then we call $\Omega$ a  \emph{simple curling}.
\begin{figure}[htbp]
\begin{center}
\includegraphics[page=2,height=5.5cm]{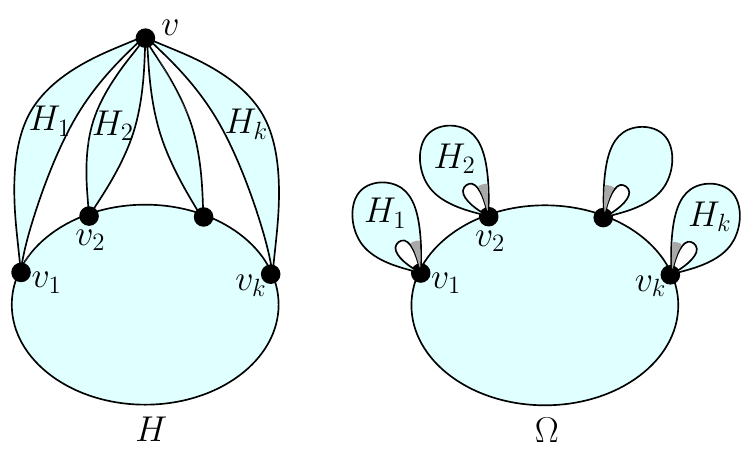}
\end{center}
\caption{A curling.} \label{fig:curling}
\end{figure}

\smallskip
\emph{4.\ Pinches}. \ If $\Omega$ is obtained from $H$ by
identifying two vertices $v_1$ and $v_2$ to  a new vertex $v$ and
labeling all edges incident with $v_1$ by $-1$ and all other edges
by $1$, then we say $\Omega$ is a \emph{pinch}. An edge with
endpoints $v_1, v_2$ becomes an unbalanced loop incident to $v$
(Figure~\ref{fig:A_pinch}).
\begin{figure}[htbp]
\begin{center}
\includegraphics[page=3,height=5cm]{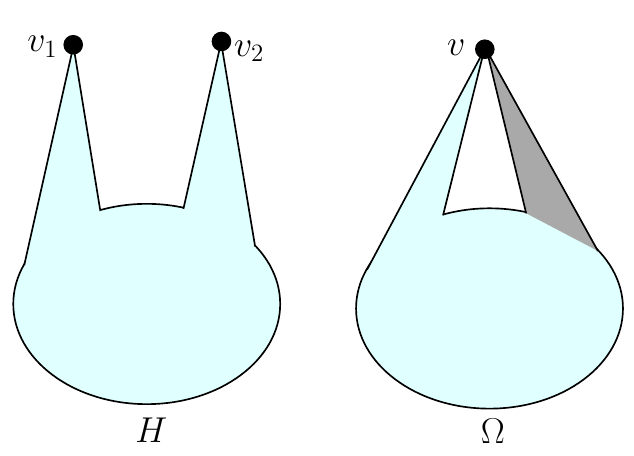}
\end{center}
\caption{A pinch.} \label{fig:A_pinch}
\end{figure}

\smallskip
\emph{5.\ 4-twistings}. \ Let $H_1,H_2,H_3,H_4$ be graphs (not
necessarily all non-empty) with distinct vertices $x_i,y_i,z_i\in
V(H_i)$. Let $H$ be obtained from $H_1,H_2,H_3,H_4$ by identifying
$x_i,y_{3-i},z_{i+2}$ to a vertex $w_i$ for every $i\in[4]$ (where
the indices are modulo 4). Let $\Omega$ be a signed graph obtained
from $H_1,H_2,H_3,H_4$ by identifying $x_1$, $x_2,x_3,x_4$ to a
vertex $x$, identifying $y_1,y_2,y_3,y_4$ to a vertex $y$ and
identifying $z_1,z_2,z_3,z_4$ to a vertex $z$, and with all edges
originally incident with $x_1, y_2$ or $z_3$ labelled by $-1$ and
all other edges labelled by $1$ (see Figure~\ref{fig:fourtwisting}).
Then we say that $\Omega$ is a \emph{4-twisting} of $H$.
\begin{figure}
\centering
\includegraphics[page=4,height=7cm]{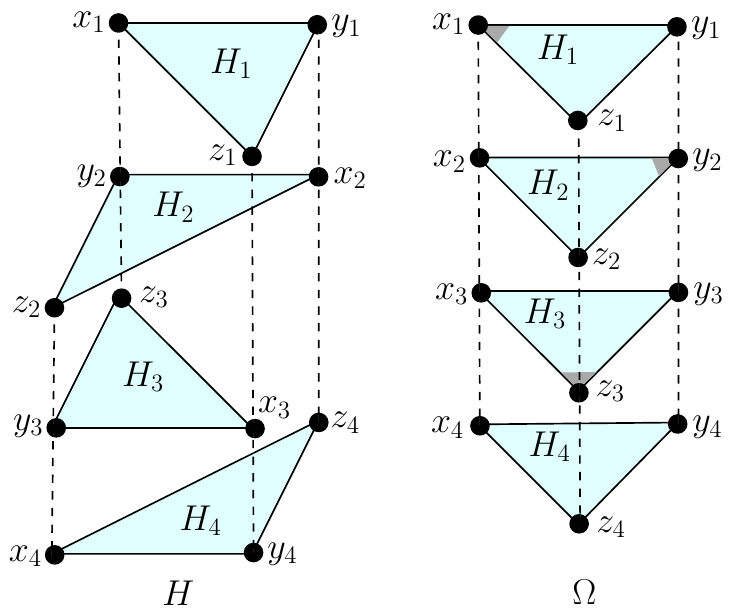}
 \caption{A 4-twisting. Vertices on a same dashed line are identified.}
 \label{fig:fourtwisting}
\end{figure}

\smallskip
\emph{6.\ Consecutive odd-twistings}. \
Let $H_1, \ldots, H_k$
(for $k \geq 3$), be graphs with distinct vertices $x_i, y_i, z_i
\in  V(H_i)$ for $i \in [k]$. Let $H$ be a graph obtained from $H_1,
\ldots, H_k$ by identifying $z_1, z_2, \ldots, z_k$ to a vertex $z$
and for each $i \in [k]$ identifying $y_{i-1}$ and $x_i$ to a vertex
$w_i$ (where the indices are modulo $k$). Let $\Omega$ be the signed
graph obtained from $H_1, \ldots, H_k$ by identifying
$y_{i-1},z_i,x_{i+1}$ to a vertex $u_i$ for every $i\in[k]$ (where
the indices are modulo $k$), and with all edges originally incident
with $y_1$ or $x_2$ labelled by $-1$ and all other edges labelled by
$1$ (see Figure~\ref{fig:cons-twisting}). Then we say that $\Omega$
is a \emph{consecutive twisting} of $H$. If $k$ is odd then $\Omega$
is a \emph{consecutive odd-twisting} of $H$.

\begin{figure}
\centering
\includegraphics[page=5,height=5cm]{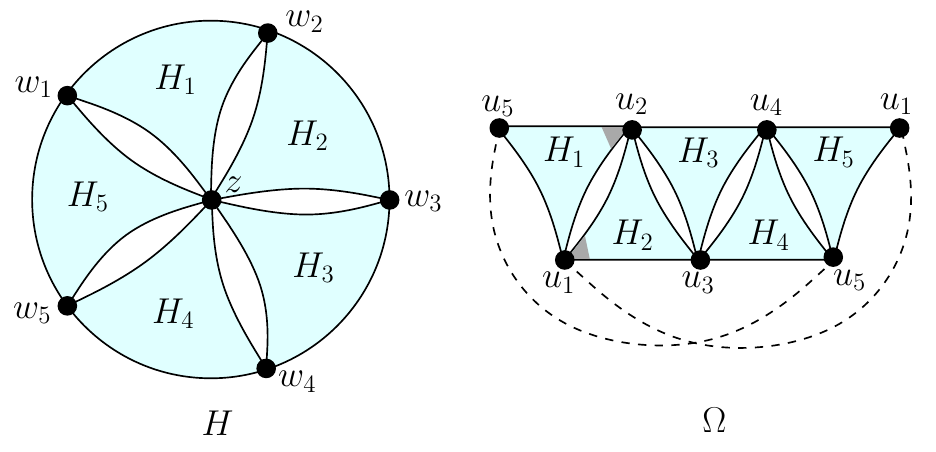}
\caption{A  consecutive odd-twisting. Vertices on a same dashed line
are identified.} \label{fig:cons-twisting}
\end{figure}

\section{All graphic frame matroids arise from these six families}
In preparation for the proof of our main result, we now introduce
some notation.
Let $H$ be a graph and $X\subseteq V(H)$.
We say $X$  is a \emph{vertex-cut} of $H$ if $H\del X$ has
at least one more component than $H$. When $|X|=1$, we also say $X$
is a \emph{cut-vertex} of $H$.
A \emph{block} of $H$ is a maximal connected
subgraph which has no cut-vertex. An \emph{ end-block} is a block
containing at most one cut-vertex.

In the rest of the paper, let $G$ be a 2-connected graph and let
$\Omega$ be a biased graph with $F(\Omega)=M(G)$. We let $\Gamma$
denote the underlying graph of $\Omega$ and $E=E(\Gamma)$.

A \emph{handcuff} consists of a pair of cycles $C_1$ and $C_2$, and
a path $P$ connecting $C_1$ and $C_2$ such that $P$ meets $C_i$ at
$u_i$ and nowhere else and $C_1$ meets $C_2$ only at
$\{u_1\}\cap\{u_2\}$. If $u_1\neq u_2$ then the handcuff is
\emph{loose}; otherwise it is \emph{tight}. A subgraph or edge set
of $\Omega$ is \emph{balanced} if each cycle in it is balanced;
otherwise it is \emph{unbalanced}. Moreover, it is
\emph{contra-balanced} if it has no balanced cycles. A vertex $v$ of
a biased graph $\Omega$ is a \emph{blocking vertex} if $\Omega\del
v$ is balanced.

Zaslavsky has characterized those biased graphs $\Omega$ for which
$F(\Omega)$ is binary.

\begin{theorem}[Zaslavsky \cite{Zaslavsky87}] \label{Zaslavsky}
Let $\Omega$ be a biased graph. Then $F(\Omega)$ is binary if and
only if each connected component of $\Omega$ has one of the
following forms.
\begin{itemize}
    \item[$(1)$] It is balanced.
    \item[$(2)$] It is a fat theta.
    \item[$(3)$] It is a signed graph with more than one unbalanced block, and each unbalanced block $B_i$ has a vertex $v_i$ such that $B_i\del v_i$ is balanced and $v_i$ is a cut-vertex separating $B_i$ from all other unbalanced blocks.
    \item[$(4)$] It is a signed graph with just one unbalanced block, and has no two vertex-disjoint unbalanced cycles.
\end{itemize}
\end{theorem}

Therefore, any biased graph $\Omega$ with graphic frame matroid has
one of the forms (1)-(4) of Theorem~\ref{Zaslavsky}. Evidently, when
$\Omega$ is balanced, by Whitney\rq{}s 2-Isomorphism Theorem
$\Gamma$ is 2-isomorphic to $G$. That is, when $\Omega$ has the form
in Theorem~\ref{Zaslavsky}(1), $\Omega$ is balanced with underlying
graph $\Gamma$ 2-isomorphic to $G$. Next we consider a biased graph
$\Omega$ that has one of forms (2)-(4) of Theorem~\ref{Zaslavsky}.

First we consider an $\Omega$ that has form
Theorem~\ref{Zaslavsky}(2). Assume that $\Omega$ is a fat theta
obtained from balanced graphs $\Omega_1,\Omega_2,\Omega_3$ by
identifying $u_1,u_2,u_3$ to a vertex $u$ and $v_1,v_2,v_3$ to a
vertex $v$, where $u_i,v_i\in V(\Gamma_i)$ (where $\Gamma_i$ is the underlying graph of $\Omega_i$).
Let $H$ be the graph
obtained from $\Gamma_1, \Gamma_2, \Gamma_3$ by
identifying $u_i$ with $v_{i+1}$ for any $i\in[3]$, where the
subscripts are modulo 3. Evidently, $M(H)=F(\Omega)$; and
consequently, by Whitney's 2-Isomorphism Theorem $H$ is 2-isomorphic
to $G$ as $F(\Omega)=M(G)$ implying $M(H)=M(G)$. So we only need to
consider $\Omega$ with forms (3) and (4) of Theorem \ref{Zaslavsky}.
These cases will be discussed in Sections~\ref{sec:form_Zas_3}
and~\ref{sec:form_Zas_4} respectively. We end this section with two
results that will be used without reference sometimes. The first one
appears in~\cite{Zaslavsky81}.

\begin{lemma}[\cite{Zaslavsky81}, Theorem $6$]\label{notheta}
A biased graph is a signed graph if and only if it has no
contra-balanced theta subgraphs.
\end{lemma}
\noindent The last result of this section is an immediate
consequence of Theorem~\ref{Zaslavsky} and Lemma~\ref{notheta}.

\begin{corollary}\label{type-2-con}
Let $\Omega$ be an unbalanced signed graph such that $F(\Omega)$ is
a connected binary matroid. Then $\Omega$ has no balanced loops and
at least one of the following holds.
\begin{itemize}
    \item[$(1)$] $\Omega$ consists of one unbalanced block.
    \item[$(2)$] $\Omega$ has more than one unbalanced blocks and a block is
unbalanced if and only if it is an end-block. Moreover, when
$F(\Omega)$ is $3$-connected, each unbalanced block is an unbalanced
loop.
\end{itemize}
\end{corollary}

\subsection{$\Omega$ with form Theorem~\ref{Zaslavsky}(3)}
\label{sec:form_Zas_3}

In this section, we mainly characterize those signed graphs $\Omega$
representing the $2$-connected graph $G$ with form Theorem
\ref{Zaslavsky}(3), that is, $\Omega$ is a signed graph with more
than one unbalanced block, and each unbalanced block $B_i$ has a
vertex $v_i$ such that $B_i\del v_i$ is balanced and $v_i$ is a
cut-vertex separating $B_i$ from all other unbalanced blocks. It
follows from Corollary~\ref{type-2-con} that a block is unbalanced
if and only if it is an end-block.

First we show that when $\Omega$ is a curling, $F(\Omega)$ is
graphic.

\begin{lemma}\label{Zasvaslky-(4)-2}
Let $\Omega$ be a curling of $H$ defined as
Section~\ref{sec:SixFamilies}. Then $M(H)=F(\Omega)$.
\end{lemma}

\begin{proof}
Let $C$ be an arbitrary cycle of $H$. When $v\notin C$, the set $C$
is also a balanced cycle of $H$. So we may assume $v\in C$ and
$e_1=vu_1, e_2=vu_2\in C$. When $u_1, u_2$ are in the same $H_i$,
$C$ is also a balanced cycle of $\Omega$; otherwise, $C$ is a
contra-balanced handcuff of $\Omega$. Therefore, every circuit of
$M(H)$ is a circuit of $F(\Omega)$.

On the other hand, let $C$ be an arbitrary circuit of $F(\Omega)$.
Evidently, $C$ is a balanced cycle or a contra-balanced handcuff of
$\Omega$ as $\Omega$ is a signed graph with no contra-balanced theta
subgraph. In either case, by the definition of $\Omega$, it is easy
to verify that $C$ is a cycle of $H$. Hence, every circuit of
$F(\Omega)$ is a circuit of $M(H)$.
\end{proof}

Secondly, we show that when $\Omega$ is a biased graph representing
$M(G)$ with more than one unbalanced block, there is a graph $H$
2-isomorphic to $G$ such that $\Omega$ is obtained as a curling of
$H$. To prove this we need some definitions and results first.

Assume that $\Omega$ is a signed graph, and $(V_1,V_2)$ is an
arbitrary partition of $V$. Let $\delta=(V_1,1;V_2,-1)$ be a
labeling of $V$ such that any vertex in $V_1$ is labelled by 1 and
any vertex in $V_2$ labelled by $-1$. Then $\delta(\Omega)$ is a
\emph{switching} of $\Omega$ with any edge relabelled by the product
of its end-vertices' labeling and its original labeling in $\Omega$.
Evidently, $F(\Omega)=F(\delta(\Omega))$.

\begin{lemma}\label{sign-bal}
Let $\Omega$ be a balanced signed graph. Then by switching all edges
of $\Omega$ can be labelled by $1$.
\end{lemma}

\begin{proof}
It suffices to show that the result holds when $\Gamma$ is
connected. Let $T$ be a spanning tree of $\Gamma$. Then for some
switching $\delta(\Omega)$, every edge of $T$ is labelled by $1$.
For every edge $e$ not in $T$, the unique cycle in $T \cup \{e\}$ is
balanced, thus $e$ is also labelled with $1$ in $\delta(\Omega)$. It
follows that all edges in $\delta(\Omega)$ are labelled by $1$.
\end{proof}

\begin{lemma}\label{a-bal-vertex-case}
Let $G$ be a $2$-connected graph and $\Omega$ be an unbalanced
signed graph with a blocking vertex $v$ and satisfying
$F(\Omega)=M(G)$. Then there is a graph $H$ 2-isomorphic to $G$ such
that $\Omega$ is obtained from $H$ by a pinch.
\end{lemma}

\begin{proof}
By Lemma~\ref{sign-bal}, it is easy to see that by some switching we
can assume that all edges of $\Omega$ labelled by $-1$ are incident
with $v$. Moreover, since $\Omega$ is unbalanced, some edges
incident with $v$ are labelled by 1 and some edges incident with $v$
are labelled by $-1$. Let $H$ be the graph obtained from $\Omega$ by
splitting $v$ into $v_1$ and $v_2$ such that any edge $e=vu$
labelled by $-1$ is changed to $e=v_1u$ and any edge $e=vu$ labelled
by 1 is changed to $e=v_2u$ and with all other edges not incident
with $v$ unchanged. Every unbalanced loop at $v$ becomes a $v_1v_2$
edge. Evidently, $F(\Omega)=M(H)$; and hence, $H$ 2-isomorphic to
$G$ as $F(\Omega)=M(G)$.
\end{proof}

A graph $H$ is a \emph{path graph} if $H$ is connected and its
blocks-cut-vertices graph is a path. The proof of Lemma
\ref{with-a-bal-vertex} is similar to the proof of Lemma
\ref{a-bal-vertex-case}.

\begin{lemma}\label{with-a-bal-vertex}
Let $\Omega$ be an unbalanced signed graph with a blocking vertex
$v$. Then there is a graph $H$ with $F(\Omega)=M(H)$ and such that
$\Gamma$ is obtained by identifying two vertices $v_1$ and $v_2$ of
$H$ to $v$. Moreover, if $\Gamma$ is $2$-connected and $H$ is a path
graph that is not $2$-connected, then each end-block contains
exactly one of $v_1$ and $v_2$.
\end{lemma}

For a path graph $H$, arbitrarily choose two vertices $v_1, v_2$
from its end-blocks such that when $H$ is not $2$-connected neither
$v_1$ nor $v_2$ is a cut-vertex of $H$ and they are not in the same
end-block. Add an edge $e$ connecting $v_1$ and $v_2$ to obtain a
new graph $H_1$. Let $H_1'$ be a graph 2-isomorphic to $H_1$ and
$v_1', v_2'$ be the end-vertices of $e$ in $H_1'$. Evidently, graph
$H'=H_1'-e$ is 2-isomorphic to $H$ and any $v_1v_2$-path in $H$ is
changed to a $v_1'v_2'$-path in $H'$ although the order of edges may
be different. In this case we say that $H'$ is a path graph
\emph{2-isomorphic to $H$ with $v_1v_2$-paths changed to
$v_1'v_2'$-paths}.

\begin{lemma}\label{Zasvaslky-(4)-1}
Let $\Omega$ be a signed graph with $\Gamma$ connected and such that
a block is unbalanced if and only if it is an end-block. Assume each
unbalanced block $B_i$ has a vertex $v_i$ such that $B_i\del v_i$ is
balanced and $v_i$ is a cut-vertex separating $B_i$ from all other
unbalanced blocks. Let $B_1,\cdots,B_k$ be all end-blocks of
$\Omega$ and for each $i\in[k]$ set
$E_i=E(B_i),\Gamma_i=\Gamma|E_i$. Then  the following hold.
\begin{itemize}
    \item[$(1)$] For some switch $\delta(\Omega)$, every edge labelled by $-1$ is in some $B_i$ and incident with $v_i$ for some $i\in[k]$.
    \item[$(2)$] For each $i\in[k]$, there is a path graph $H_i'$ such that $B_i$ is obtained from $H_i'$ by identifying
$v_{i1}'$ and $v_{i2}'$ to $v_i$ with all edges originally incident
with $v_{i1}'$ labelled by $-1$ and all other edges not incident
with $v_{i1}'$ labelled by $1$ and satisfying $F(B_i)=M(H_i')$.
    \item[$(3)$] For each $i\in[k]$, let $H_i$ be a path graph 2-isomorphic to
$H_i'$ with $v_{i1}'v_{i2}'$-paths changed to $v_{i1}v_{i2}$-paths.
First add a new isolated vertex $v$ to the graph
$\Gamma'=\Gamma\del(E_1\cup\cdots \cup E_k)$, and then add
$H_1,\cdots, H_k$ to $\Gamma'$ by identifying $v_{i2}$ with $v_i$
and $v_{11},\cdots, v_{k1}$ with $v$. Let $H$ denote the new graph.
Then $F(\Omega)=M(H)$.
\end{itemize}
\end{lemma}

\begin{proof}
Evidently, (1) is an immediate consequence of Lemma~\ref{sign-bal}.
Moreover, since each $B_i$ is a signed graph with a blocking vertex
$v_i$, (2) follows immediately from Lemma \ref{with-a-bal-vertex}.
To show (3), let $C$ be an arbitrary cycle of $H$. When
$C\cap(E_1\cup\cdots\cup E_k)=\emptyset$, the set $C$ is also a
balanced cycle of $\Omega$. When $C\subseteq E_i$ for some
$i\in[k]$, since $H_i$ is 2-isomorphic to $H_i'$ and
$F(B_i)=M(H_i')$, the set $C$ is a circuit of $F(B_i)$. So we may
assume that $C\cap E_i, C\cap(E\del E_i)\neq\emptyset$ for some
$i\in[k]$. Evidently, there is only one integer $i\neq j\in[k]$ such
that $C\cap E_j\neq\emptyset$, and for any $s\in\{i,j\}$, the set
$C\cap E_s$ is a $v_{s1}v_{s2}$-path of $H_s$; and consequently,
$C\cap E_s$ is a $v_{s1}'v_{s2}'$-path of $H_s'$ as $H_s$ is a path
graph 2-isomorphic to $H_s'$ with $v_{s1}'v_{s2}'$-paths changed to
$v_{s1}v_{s2}$-paths. Thus, by (1) and (2) $C$ is a contra-balanced
handcuff of $\Omega$. So every circuit of $M(H)$ is a circuit of
$F(\Omega)$.

On the other hand, assume that $C$ is an arbitrary circuit of
$F(\Omega)$. Then $C$ is a balanced cycle or a contra-balanced
handcuff of $\Omega$ as $\Omega$ is a signed-graph. When $C$ is a
balanced cycle, no matter whether $C\subseteq E_i$ or
$C\cap(E_1\cup\cdots\cup E_k)=\emptyset$, by the definition of $H$
it is easy to see that the set $C$ is also a cycle of $H$ as $H_i$
is 2-isomorphic to $H_i'$. So we may assume that $C$ is a
contra-balanced handcuff of $\Omega$. Without loss of generality we
may assume $C\cap E_i, C\cap E_j\neq\emptyset$. Then for any
$s\in\{i,j\}$, the set $C\cap E_s$ is a $v_{s1}v_{s2}$-path of $H_s$
as $H_s$ is a path graph 2-isomorphic to $H_s'$ with
$v_{s1}'v_{s2}'$-paths changed to $v_{s1}v_{s2}$-paths. Therefore,
$C$ is also a cycle of $H$. So every circuit of $F(\Omega)$ is a
circuit of $M(H)$.
\end{proof}

Therefore, when $\Omega$ has form Theorem~\ref{Zaslavsky}(3), it
follows from Lemma~\ref{Zasvaslky-(4)-1} that $\Omega$ is obtained
as a curling. Moreover, when $G$ is $3$-connected, by
Corollary~\ref{type-2-con} each unbalanced block of $\Omega$ is a
loop. Thus, by Lemma~\ref{Zasvaslky-(4)-1} we have the following
result.

\begin{corollary}\label{Zalavsky-(4)-3-con}
Let $G$ be a $3$-connected graph and let $\Omega$ be an unbalanced
signed graph with $F(\Omega)=M(G)$. Assume that $\Omega$ has more
than one unbalanced block. Then $G$  is obtained from $\Gamma$ by
adding a new isolated vertex $v$ to $\Gamma$ and changing all loops
to links connecting $v$ and their original end-vertices.
\end{corollary}

\subsection{$\Omega$ with form Theorem~\ref{Zaslavsky}(4)}
\label{sec:form_Zas_4}

In this section, we mainly characterize the signed graphs $\Omega$
representing the $2$-connected graph $G$ with form Theorem
\ref{Zaslavsky}(4). These have just one unbalanced block and no two
vertex-disjoint unbalanced cycles.

While Slilaty~\cite{MR2344133} characterized those signed graphs having no
blocking vertex and no two vertex disjoint unbalanced cycles having
graphic frame matroid in terms of projective-planar signed graphs
and 1, 2, and 3-sums of balanced signed graphs, an application of a
theorem on lift matroids gives us a different structural
characterization. The \emph{lift matroid} $L(\Omega)$ of a signed
graph $\Omega$ was defined by Zaslavsky in \cite{Zaslavsky-II}. Its
circuits are the sets of edges of one of the following two types:
balanced cycles and the union of two unbalanced cycles meeting in at
most one vertex. In his Ph.D. thesis Shih proved the following
characterisation of graphic lift matroids (see also \cite{Pivotto},
Theorem 4.1).

\begin{theorem}[Theorem 1, Chapter 2 in \cite{Shih}]\label{Shih-Pivotto}
Let $G$ be a graph and let $\Omega$ be a signed graph such that
$M(G)=L(\Omega)$. Then there exists a graph $H$ $2$-isomorphic to
$G$ such that one of the following holds.
\begin{itemize}
    \item[$(1)$]  $\Omega$ is obtained from $H$ by a pinch.
    \item[$(2)$]  $\Omega$ is obtained from $H$ by a 4-twisting.
    \item[$(3)$] $\Omega$ is obtained from $H$ by a consecutive twisting.
\end{itemize}
\end{theorem}

Since $L(\Omega)=F(\Omega)$ when $\Omega$ has no vertex-disjoint
unbalanced cycles, the signed graph we want to find consisting of
one unbalanced block without vertex-disjoint unbalanced cycles has
the form of one of Theorem~\ref{Shih-Pivotto}(1)-(3). However, the
signed graph $\Omega$ in Theorem~\ref{Shih-Pivotto}(3) may have
vertex-disjoint unbalanced cycles, so we only need to find all
signed graphs having no vertex-disjoint unbalanced cycles.
Evidently, when $\Omega$ is obtained through~\ref{Shih-Pivotto}(1),
that is, obtained as a pinch, $\Omega$ has no two vertex-disjoint
unbalanced cycles. On the other hand, note that $\Omega$ has no
vertex-disjoint unbalanced cycles if and only if each cycle of $G$
is connected in $\Gamma$. Thus, we only need to determine under
which conditions a cycle of $H$ is connected in $\Gamma$, for the
graph $H$ (2-isomorphic to $G$) given in Theorem~\ref{Shih-Pivotto}.

\begin{lemma}\label{Case2}
Suppose that $\Omega$ is obtained from $H$ by a 4-twisting as in
Theorem~\ref{Shih-Pivotto}(2). Then every cycle of $H$ is connected
in $\Gamma$.
\end{lemma}

\begin{proof}
Let $C$ be an arbitrary cycle of $H$. Assume to the contrary that
$C$ is not connected in $\Gamma$. Then $C$ is  a union of two
vertex-disjoint cycles $C_1$ and $C_2$ of $\Gamma$ as
$M(G)=L(\Omega)$. Moreover, by the definition of $\Gamma$, either
$|C_1\cap\{x,y,z\}|=1$ or $|C_2\cap\{x,y,z\}|=1$. By symmetry we may
assume that the former holds. Then $C_1$ is a cycle of $H$, a
contradiction to the fact that $C_1$ is a proper subset of the cycle
$C$ of $H$.
\end{proof}

\begin{lemma}\label{Case3-2-con}
Suppose that $\Omega$ is obtained from $H$ by a consecutive twisting
as in Theorem~\ref{Shih-Pivotto}(3). If $G$ is $2$-connected, then
every cycle in $H$ is connected in $\Gamma$ if and only if for some
$i\in[k]$ no path connects $x_i$ and $y_i$ in $H_i\del z_i$ when $k$
is even.
\end{lemma}

\begin{proof}
First we prove the ``only if" part. Assume to the contrary that
$k=2n$ and for any $i\in[k]$
there is a path connecting $x_i$ and $y_i$ in $H_i\del z_i$. 
Then $G'$ has a cycle $C=P_{x_1,y_1}P_{x_2,y_2}\cdots P_{x_k,y_k}$,
where $P_{x_i,y_i}$ is a path of $H_i$ connecting $x_i$ and $y_i$
with $z_i\notin P_{x_i,y_i}$. However,
$C_1=P_{x_1,y_1}P_{x_3,y_3}\cdots P_{x_{2n-1},y_{2n-1}}$ and
$C_2=P_{x_2,y_2}P_{x_4,y_4}\cdots P_{x_{2n},y_{2n}}$ are
vertex-disjoint cycles of $\Gamma$ such that $(E(C_1), E(C_2))$ is a
partition of $E(C)$, a contradiction.

Secondly, we prove the ``if" part. Let $C$ be an arbitrary cycle of
$H$. Evidently, when $C$ is completely contained in some $H_i$, the
set $C$ is also a cycle of $\Gamma$. So we may assume that $C$
intersects at least two $H_i$'s. Assume that at most one $H_i|C$
uses $z_i$. Under this case $C$ must have the structure
$P_{x_1,y_1}P_{x_2,y_2}\cdots P_{x_k,y_k}$, where $P_{x_i,y_i}$ is a
path of $H_i$ connecting $x_i$ and $y_i$; and hence, either $k=2n+1$
or $k=2n$ and there is exactly one integer $i\in[k]$ with $z_i\in
P_{x_i,y_i}$, say $i=1$. When $k=2n+1$, we have
$$C'=P_{x_1,y_1}P_{x_3,y_3}\cdots
P_{x_{2n+1},y_{2n+1}}P_{x_2,y_2}P_{x_4,y_4}\cdots
P_{x_{2n},y_{2n}}$$ is a cycle of $\Gamma$ with $E(C)=E(C')$ if $z_1
\notin P_{x_1,y_1}$ and $C'$ is the union of two edge-disjoint
cycles of $\Gamma$ sharing vertex $z_1$ otherwise. When $k=2n$ (so
$z_1\in P_{x_1,y_1}$), we have $C_1=P_{x_1,y_1}P_{x_3,y_3}\cdots
P_{x_{2n-1},y_{2n-1}}$ and $C_2=P_{x_2,y_2}P_{x_4,y_4}\cdots
P_{x_{2n},y_{2n}}$ are two edge-disjoint cycles of $\Gamma$ with a
unique common vertex $u_1$. Hence, we can assume that there are
exactly two $H_i|C$ using $z_i$. Without loss of generality we may
assume that $C=P_{z_1,y_1}P_{x_2,y_2}\cdots P_{x_m,z_m}$, where
$P_{z_1,y_1},P_{x_i,y_i}(2\leq i\leq m-1),P_{x_m,z_m}$ are paths of
$H_1,H_i,H_m$, respectively. When $m=2s$, the set
$$C'=P_{z_1,y_1}P_{x_3,y_3}\cdots P_{x_{2s-1},y_{2s-1}}P_{z_{2s},x_{2s}}P_{y_{2s-2},x_{2s-2}}\cdots P_{y_2,x_2}$$
is a cycle of $\Gamma$ with $E(C)=E(C')$; and when $m=2s+1$, the set
$$C'=P_{z_1,y_1}P_{x_3,y_3}\cdots P_{x_{2s-1},y_{2s-1}}P_{x_{2s+1},z_{2s+1}}P_{y_{2s},x_{2s}}P_{y_{2s-2},x_{2s-2}}\cdots P_{y_2,x_2}$$
is a cycle of $\Gamma$ with $E(C)=E(C')$.
\end{proof}

\begin{lemma}\label{Case-k-even}
Suppose that $\Omega$ is obtained from $H$ by a consecutive
twisting, where $k$ is even. If $G$ is $2$-connected and every cycle
of $H$ is connected in $\Gamma$, then $\Omega$ is a pinch.
\end{lemma}

\begin{proof}
By Lemma~\ref{Case3-2-con}, for some $i\in[k]$ no path connects
$x_i$ and $y_i$ in $H_i\del z_i$. Then every unbalanced cycle in
$\Omega$ uses $z_i$, i.e. $z_i$ is a blocking vertex of $\Omega$.
Lemma~\ref{a-bal-vertex-case} implies that  there is a graph $H'$
2-isomorphic to $G$ such that $\Omega$ is obtained from $H'$ by a
pinch.
\end{proof}

Therefore, by Lemmas~\ref{Case2}, \ref{Case3-2-con}
and~\ref{Case-k-even}, and the analysis in the paragraph following
Theorem~\ref{Shih-Pivotto}, those signed graphs $\Omega$ having
$F(\Omega)$ isomorphic to $M(G)$ with form
Theorem~\ref{Zaslavsky}(4) (that is, having just one unbalanced
block and without two vertex-disjoint unbalanced cycles) are
obtained by a pinch, a 4-twisting, or a consecutive odd-twisting.

\section*{Acknowledgements}
We thank Geoff Whittle for helpful discussions. And we also thanks
the referees for carefully reading the paper and suggesting many
improvements.

\end{document}